\documentclass[12pt,reqno]{amsart}
\usepackage{amscd,amssymb,amsmath,multicol,float,scalefnt,cancel,array,stmaryrd,graphicx,tikz,tikz-cd}
\usepackage[right=1.5in]{geometry}
\restylefloat{table}
\begin{document}

\newtheorem{thm}[equation]{Theorem}
\numberwithin{equation}{section}
\newtheorem{cor}[equation]{Corollary}
\newtheorem{expl}[equation]{Example}
\newtheorem{rmk}[equation]{Remark}
\newtheorem{conv}[equation]{Convention}
\newtheorem{claim}[equation]{Claim}
\newtheorem{lem}[equation]{Lemma}
\newtheorem{sublem}[equation]{Sublemma}
\newtheorem{conj}[equation]{Conjecture}
\newtheorem{defin}[equation]{Definition}
\newtheorem{diag}[equation]{Diagram}
\newtheorem{prop}[equation]{Proposition}
\newtheorem{notation}[equation]{Notation}
\newtheorem{tab}[equation]{Table}
\newtheorem{fig}[equation]{Figure}
\newcounter{bean}
\renewcommand{\theequation}{\thesection.\arabic{equation}}

\raggedbottom \voffset=-.7truein \hoffset=0truein \vsize=8truein
\hsize=6truein \textheight=8truein \textwidth=6truein
\baselineskip=18truept

\def\mapright#1{\ \smash{\mathop{\longrightarrow}\limits^{#1}}\ }
\def\mapleft#1{\smash{\mathop{\longleftarrow}\limits^{#1}}}
\def\mapup#1{\Big\uparrow\rlap{$\vcenter {\hbox {$#1$}}$}}
\def\mapdown#1{\Big\downarrow\rlap{$\vcenter {\hbox {$\ssize{#1}$}}$}}
\def\mapne#1{\nearrow\rlap{$\vcenter {\hbox {$#1$}}$}}
\def\mapse#1{\searrow\rlap{$\vcenter {\hbox {$\ssize{#1}$}}$}}
\def\mapr#1{\smash{\mathop{\rightarrow}\limits^{#1}}}
\def\ss{\smallskip}
\def\s{\sigma}
\def\l{\lambda}
\def\vp{v_1^{-1}\pi}
\def\at{{\widetilde\alpha}}

\def\sm{\wedge}
\def\la{\langle}
\def\ra{\rangle}
\def\ev{\text{ev}}
\def\od{\text{od}}
\def\on{\operatorname}
\def\ol#1{\overline{#1}{}}
\def\spin{\on{Spin}}
\def\cat{\on{cat}}
\def\Lbar{\overline{\Lambda}}
\def\qed{\quad\rule{8pt}{8pt}\bigskip}
\def\ssize{\scriptstyle}
\def\a{\alpha}
\def\bz{{\Bbb Z}}
\def\Rhat{\hat{R}}
\def\im{\on{im}}
\def\ct{\widetilde{C}}
\def\ext{\on{Ext}}
\def\sq{\on{Sq}}
\def\eps{\epsilon}
\def\ar#1{\stackrel {#1}{\rightarrow}}
\def\br{{\bold R}}
\def\bC{{\bold C}}
\def\bA{{\bold A}}
\def\bB{{\bold B}}
\def\bD{{\bold D}}
\def\bC{{\bold C}}
\def\bh{{\bold H}}
\def\bQ{{\bold Q}}
\def\bP{{\bold P}}
\def\bx{{\bold x}}
\def\bo{{\bold{bo}}}
\def\dh{\widehat{d}}
\def\si{\sigma}
\def\Vbar{{\overline V}}
\def\dbar{{\overline d}}
\def\wbar{{\overline w}}
\def\Sum{\sum}
\def\tfrac{\textstyle\frac}

\def\tb{\textstyle\binom}
\def\Si{\Sigma}
\def\w{\wedge}
\def\equ{\begin{equation}}
\def\b{\beta}
\def\G{\Gamma}
\def\L{\Lambda}
\def\g{\gamma}
\def\d{\delta}
\def\k{\kappa}
\def\psit{\widetilde{\Psi}}
\def\tht{\widetilde{\Theta}}
\def\psiu{{\underline{\Psi}}}
\def\thu{{\underline{\Theta}}}
\def\aee{A_{\text{ee}}}
\def\aeo{A_{\text{eo}}}
\def\aoo{A_{\text{oo}}}
\def\aoe{A_{\text{oe}}}
\def\vbar{{\overline v}}
\def\endeq{\end{equation}}
\def\sn{S^{2n+1}}
\def\zp{\bold Z_p}
\def\cR{{\mathcal R}}
\def\P{{\mathcal P}}
\def\cQ{{\mathcal Q}}
\def\cj{{\cal J}}
\def\zt{{\bold Z}_2}
\def\bs{{\bold s}}
\def\bof{{\bold f}}
\def\bq{{\bold Q}}
\def\be{{\bold e}}
\def\Hom{\on{Hom}}
\def\ker{\on{ker}}
\def\kot{\widetilde{KO}}
\def\coker{\on{coker}}
\def\da{\downarrow}
\def\colim{\operatornamewithlimits{colim}}
\def\zphat{\bz_2^\wedge}
\def\io{\iota}
\def\om{\omega}
\def\Prod{\prod}
\def\e{{\cal E}}
\def\zlt{\Z_{(2)}}
\def\exp{\on{exp}}
\def\abar{{\overline a}}
\def\xbar{{\overline x}}
\def\ybar{{\overline y}}
\def\zbar{{\overline z}}
\def\mbar{{\overline m}}
\def\nbar{{\overline n}}
\def\sbar{{\overline s}}
\def\kbar{{\overline k}}
\def\bbar{{\overline b}}
\def\et{{\widetilde E}}
\def\ni{\noindent}
\def\tsum{\textstyle \sum}
\def\coef{\on{coef}}
\def\den{\on{den}}
\def\lcm{\on{l.c.m.}}
\def\vi{v_1^{-1}}
\def\ot{\otimes}
\def\psibar{{\overline\psi}}
\def\thbar{{\overline\theta}}
\def\mhat{{\hat m}}
\def\exc{\on{exc}}
\def\ms{\medskip}
\def\ehat{{\hat e}}
\def\etao{{\eta_{\text{od}}}}
\def\etae{{\eta_{\text{ev}}}}
\def\dirlim{\operatornamewithlimits{dirlim}}
\def\gt{\widetilde{L}}
\def\lt{\widetilde{\lambda}}
\def\st{\widetilde{s}}
\def\ft{\widetilde{f}}
\def\sgd{\on{sgd}}
\def\lfl{\lfloor}
\def\rfl{\rfloor}
\def\ord{\on{ord}}
\def\gd{{\on{gd}}}
\def\rk{{{\on{rk}}_2}}
\def\nbar{{\overline{n}}}
\def\MC{\on{MC}}
\def\lg{{\on{lg}}}
\def\cH{\mathcal{H}}
\def\cS{\mathcal{S}}
\def\cP{\mathcal{P}}
\def\N{{\Bbb N}}
\def\Z{{\Bbb Z}}
\def\Q{{\Bbb Q}}
\def\R{{\Bbb R}}
\def\C{{\Bbb C}}
\def\Lb{\overline\Lambda}
\def\mo{\on{mod}}
\def\xt{\times}
\def\notimm{\not\subseteq}
\def\Remark{\noindent{\it  Remark}}
\def\kut{\widetilde{KU}}
\def\Eb{\overline E}
\def\*#1{\mathbf{#1}}
\def\0{$\*0$}
\def\1{$\*1$}
\def\22{$(\*2,\*2)$}
\def\33{$(\*3,\*3)$}
\def\ss{\smallskip}
\def\ssum{\sum\limits}
\def\dsum{\displaystyle\sum}
\def\la{\langle}
\def\ra{\rangle}
\def\on{\operatorname}
\def\proj{\on{proj}}
\def\od{\text{od}}
\def\ev{\text{ev}}
\def\o{\on{o}}
\def\U{\on{U}}
\def\lg{\on{lg}}
\def\a{\alpha}
\def\bz{{\Bbb Z}}
\def\eps{\varepsilon}
\def\bc{{\bold C}}
\def\bN{{\bold N}}
\def\bB{{\bold B}}
\def\bW{{\bold W}}
\def\nut{\widetilde{\nu}}
\def\tfrac{\textstyle\frac}
\def\b{\beta}
\def\G{\Gamma}
\def\g{\gamma}
\def\zt{{\Bbb Z}_2}
\def\zth{{\bold Z}_2^\wedge}
\def\bs{{\bold s}}
\def\bx{{\bold x}}
\def\bof{{\bold f}}
\def\bq{{\bold Q}}
\def\be{{\bold e}}
\def\lline{\rule{.6in}{.6pt}}
\def\xb{{\overline x}}
\def\xbar{{\overline x}}
\def\ybar{{\overline y}}
\def\zbar{{\overline z}}
\def\ebar{{\overline \be}}
\def\nbar{{\overline n}}
\def\ubar{{\overline u}}
\def\bbar{{\overline b}}
\def\et{{\widetilde e}}
\def\lf{\lfloor}
\def\rf{\rfloor}
\def\ni{\noindent}
\def\ms{\medskip}
\def\Dhat{{\widehat D}}
\def\what{{\widehat w}}
\def\Yhat{{\widehat Y}}
\def\abar{{\overline{a}}}
\def\minp{\min\nolimits'}
\def\sb{{$\ssize\bullet$}}
\def\mul{\on{mul}}
\def\N{{\Bbb N}}
\def\Z{{\Bbb Z}}
\def\S{\Sigma}
\def\Q{{\Bbb Q}}
\def\R{{\Bbb R}}
\def\C{{\Bbb C}}
\def\Xb{\overline{X}}
\def\eb{\overline{e}}
\def\notint{\cancel\cap}
\def\cS{\mathcal S}
\def\cR{\mathcal R}
\def\el{\ell}
\def\TC{\on{TC}}
\def\GC{\on{GC}}
\def\wgt{\on{wgt}}
\def\Ht{\widetilde{H}}
\def\wbar{\overline w}
\def\dstyle{\displaystyle}
\def\Sq{\on{sq}}
\def\Om{\Omega}
\def\ds{\dstyle}
\def\tz{tikzpicture}
\def\zcl{\on{zcl}}
\def\bd{\bold{d}}
\def\io{\iota}
\def\Vb#1{{\overline{V_{#1}}}}
\def\Ebar{\overline{E}}
\def\lb{$\scriptstyle\bullet$}
\def\Cc{\mathcal{C}}

\title
{Duality in $BP\la n\ra$  (co)homology}
\author{Donald M. Davis}
\address{Department of Mathematics, Lehigh University\\Bethlehem, PA 18015, USA}
\email{dmd1@lehigh.edu}

\date{May 23, 2022}

\keywords{Universal Coefficient Theorem, Brown-Peterson(co)homology, Ext groups, K theory}
\thanks {2000 {\it Mathematics Subject Classification}: 55U20, 55U30, 55N20, 18G15.}

\maketitle

\begin{abstract} Let $E=BP\la n\ra$ denote the Johnson-Wilson spectrum, localized at $p$. It is proved that if $E_*(X)$ is locally finite, then there is an isomorphism of right $E_*$-modules $E^*(X)\approx (E_*(\Sigma^{D+n+1}X))^\vee$, where $D=\sum|v_i|$ and $M^\vee=\Hom(M,\Q/\Z)$ is the Pontryagin dual.
 This result was motivated by work of the author and W.S.Wilson regarding the 2-local $ku$-homology and -cohomology groups of the Eilenberg-MacLane space $K(\Z/2,2)$.
\end{abstract}

\section{Main results}\label{intro} Let $E=BP\la n\ra$ denote the  Johnson-Wilson spectrum (\cite{JW}) localized at a prime $p$, which satisfies that
$E_*=\pi_*(E)=\Z_{(p)}[v_1,\ldots,v_n]$, with $|v_i|=2(p^i-1)$. Our motivating example is the case $p=2$, $n=1$, when $E$ is the spectrum $ku$ for connective complex $K$-theory, localized at 2. Our main result is an isomorphism between certain $E$-cohomology groups and the Pontryagin dual of $E$-homology groups. We require that $E_*(X)$ is locally finite, which means that for each $i$, the $E_*$-module generated by $E_i(X)$ is finite. If $M$ is an $R$-module, we denote by $M^\vee$ the right $R$-module $\Hom(M,\Q/\Z)$. Localized at $p$, we prefer to write $\Q/\Z$ as $\Z/p^\infty$.

\begin{thm}\label{main} If $E=BP\la n\ra$ and $E_*(X)$ is locally finite,
 there is an isomorphism of right $E_*$-modules
$$E^*(X)\approx (E_*(\Sigma^{D+n+1}X))^\vee,$$
where $D=\sum|v_i|=2((p^{n+1}-1)/(p-1)-(n+1))$.\end{thm}

We prove this result using a Universal Coefficient Theorem and the following algebraic result, which we prove in Section \ref{algsec}. If $M$ is a graded module, $\Sigma^DM$ denotes the graded module obtained from $M$ by increasing gradings by $D$. Our Ext groups are in the category of graded modules, and the second superscript refers to the grading.
\begin{thm}\label{algthm} Let $R=\Z_{(p)}[x_1,\ldots,x_n]$ with $|x_i|$ positive  integers, and let $D=\sum|x_i|$.  If $M$ is a locally finite graded $R$-module, there is an isomorphism of graded right $R$-modules
$$\ext_R^{s}(M,R)\approx\begin{cases}\Sigma^DM^\vee&s=n+1\\
0&s\ne n+1.\end{cases}$$
\end{thm}

\begin{proof}[Proof of Theorem \ref{main}] By \cite[Corollary, p.257]{Rob}, if $E$ is an $A_\infty$ ring spectrum, there is a Universal Coefficient spectral sequence
$$\ext_{E_*}^{s,t}(E_*X,E_*)\Rightarrow E^{s+t}X.$$
By \cite[Corollary 3.5]{BJ}, $BP\la n\ra$ is an $A_\infty$ ring spectrum. By Theorem \ref{algthm} with $R=E_*$, the spectral sequence must collapse, as it is confined to a single value of $s$, and the $E_\infty$ groups are as claimed.\end{proof}

In Section \ref{explsec}, we illustrate Theorem \ref{main} for a portion of $ku_*(K_2)$ with $K_2=K(\Z/2,2))$, localized at 2. Here we state the application of Theorem \ref{main} to this case as a corollary.
\begin{cor} There is an isomorphism of right $ku_*$-modules $ku^*(K_2)\approx (ku_*(\Sigma^4K_2))^\vee$.\end{cor}

Observe also that the case $n=0$ of Theorem \ref{main} is the usual Universal Coefficient Theorem when $H_*(X;\Z_{(p)})$ is finite.

After a version of this paper was placed on the {\tt arXiv}, John Greenlees pointed out to the author that Theorem \ref{main} could apparently be deduced using concepts of duality in stable homotopy theory, and quickly prepared a short manuscript (\cite{Gr}) which did so, at least when $n\le2$. A result using Brown-Comenetz duality (\cite[Corollary 9.3]{MR}) is closely related. We feel that the elementary nature of our presentation lends worth to our paper.

The author is grateful to Andy Baker, John Greenlees, Andrey Lazarev, Doug Ravenel, Chuck Weibel, and Steve Wilson for helpful suggestions.

\section{Proof of Theorem \ref{algthm}}\label{algsec}

The following result is certainly well-known. Here $\Z_p$ denotes the $p$-adic integers.

\begin{prop}\label{prop} Let $R=\Z_{(p)}[x_1,\ldots,x_n]$ with $|x_i|$ positive  integers, and let $D=\sum|x_i|$. In the category of graded $R$-modules
\begin{equation}\label{1}\ext_R^{s,t}(\Z/p,R)=\begin{cases}\Z/p&(s,t)=(n+1,D)\\ 0&\text{otherwise},\end{cases}\end{equation}
and
\begin{equation}\label{2}\ext_R^s(\Z/p^\infty,R)=\begin{cases}\Z_p&(s,t)=(n+1,D)\\ 0&\text{otherwise}.\end{cases}\end{equation}\end{prop}

\begin{proof} Let $\Cc_0$ be the chain complex $C_1\to C_0$ with $C_1$ and $C_0$ free $\Z_{(p)}$-modules of rank 1 and grading 0 with generators $g_0$ and $\io_0$, respectively, and $d(g_0)=p\io_0$. For $1\le i\le n$, let $\Cc_i$ be the chain complex $C_{i,1}\to C_{i,0}$ with $C_{i,1}$ and $C_{i,0}$ free $\Z_{(p)}[x_i]$-modules of rank 1 with generators $g_i$ and $\io_i$, respectively, and $d(g_i)=x_i\io_i$. Here $|\io_i|=0$ and  $|g_i|=|x_i|$. Then $\bC:=\Cc_0\ot \Cc_1\ot\cdots\ot \Cc_n$ is a chain complex of free $R$-modules with $H_j(\bC)=\Z/p$ for $j=0$, and 0 for $j>0$, by the K\"unneth Theorem. Thus $\bC$ is an $R$-resolution of $\Z/p$. Hence $\ext_R^s(\Z/p,R)$ is the $s^{\text{th}}$ cohomology group of the dual complex $\Hom_R(\bC,R)$, which is the tensor product, $\Cc^*_0\ot\Cc^*_i\ot\cdots\ot \Cc_n^*$, of the dual complexes. The cohomology group is nonzero only when $s=n+1$, where it is $\Z/p$, dual to $g_0\ot g_1\ot\cdots\ot g_n$, in grading $D$.

For the second result, we replace $\Cc_0$ by a chain complex $\Cc'$ which has $C_1'$ and $C_0'$ free $\Z_{(p)}$-modules with generators indexed by positive integers, $g'_j$ and $\io'_j$, respectively, with $d(g'_j)=\io'_j-p\io'_{j+1}$. Then $H_0(\Cc')=\Z/p^\infty$ is the nonzero homology group, and $H^1(\Cc')=\Z_p$ is the nonzero cohomology group. The rest of the proof follows as in the previous paragraph.
\end{proof}

\begin{proof}[Proof of Theorem \ref{algthm}] We first consider the case when $M$ is finite, and proceed by induction on the size of $M$. The result is true when $M=\Z/p$ by (\ref{1}). Let $\a$ denote a generator of $\ext_R^{n+1,D}(\Z/p^\infty,R)$ from (\ref{2}). Yoneda product $\a\circ$ is a natural transformation of right $R$-modules
$$\ext_R^{*,*}(-,\Z/p^\infty)\to \ext_R^{*+n+1,*+D}(-,R).$$
If $$0\to K\to M\to Q\to 0$$
is a short exact sequence of finite $R$-modules, by induction we may assume the theorem is true for $K$ and $Q$, and hence by the exact Ext sequence, $\ext_R^s(M,R)=0$ if $s\ne n+1$. We also obtain a commutative diagram of short exact sequences of right $R$-modules

\begin{tikzcd}
0\arrow{r}&\Sigma^DQ^\vee\arrow{d}\arrow{r}&\Sigma^DM^\vee\arrow{d}\arrow{r}&\Sigma^DK^\vee\arrow{d}\arrow{r}&0\\
0\arrow{r}&\ext_R^{n+1}(Q,R)\arrow{r}&\ext_R^{n+1}(M,R)\arrow{r}&\ext_R^{n+1}(K,R)\arrow{r}&0.
\end{tikzcd}

\ni The 0's on the ends of the first sequence follow from \cite[p.70]{Wei}, and for the second sequence by the induction. By the 5-lemma, our result is true for finite $R$-modules.

Now let $M$ be locally finite, and for any positive integer $k$, let $K_k$ (resp.~$Q_k$) denote the set of all elements of $M$ in grading $>k$ (resp.~$\le k$). There is a short exact sequence of $R$-modules
$$0\to K_k\to M\to Q_k\to 0.$$
Since $Q_k$ is finite, the induced $\ext_R(-,R)$ sequence implies that for $s\ne n+1$, $\ext_R^{s,j}(M,R)=0$ for $j\le k$. Since $k$ was arbitrary, we deduce that $\ext_R^s(M,R)=0$ for $s\ne n+1$. Again Yoneda product with $\a$ yields a commutative diagram of short exact sequences of right $R$-modules

\begin{tikzcd}
0\arrow{r}&\Sigma^DQ_k^\vee\arrow{d}\arrow{r}&\Sigma^DM^\vee\arrow{d}\arrow{r}&\Sigma^DK_k^\vee\arrow{d}\arrow{r}&0\\
0\arrow{r}&\ext_R^{n+1}(Q_k,R)\arrow{r}&\ext_R^{n+1}(M,R)\arrow{r}&\ext_R^{n+1}(K_k,R)\arrow{r}&0.
\end{tikzcd}

The left vertical arrow is iso since $Q_k$ is finite, and the groups in the right vertical arrow are 0 in grading $\le k$. Since $k$ is arbitrary, we deduce that the center vertical arrow is an isomorphism.
\end{proof}

\section{An example when $n=1$, $p=2$, and $X=K(\Z/2,2)$.}\label{explsec}
In \cite{W} and \cite{DW}, the author and, previously, Wilson initiated a partial calculation of $ku_*(K_2)$, where $K_2=K(\Z/2,2)$, in their studies of Stiefel-Whitney classes. In \cite{DW2}, these authors made a complete calculation of $ku^*(K_2)$. Using our new Theorem \ref{main}, we can now give a complete determination of $ku_*(K_2)$, since we know that it is locally finite, as it was noted in \cite{DW} that it contains no infinite groups or infinite $v_1$-towers.

The work in \cite{DW} and \cite{DW2} was done using the Adams spectral sequence. It is interesting to compare the forms of the two Adams spectral sequence $E_\infty$ calculations. What appears as an $h_0$ multiplication in one usually appears as an exotic extension (multiplication by 2 not seen in Ext) in the other. We illustrate here with corresponding small portions of each. The portion of $ku^*(K_2)$ in Figure \ref{fig1} is called $A_5$ in \cite{DW2}. Note that in our $ku^*$ chart, indices increase from right to left. Exotic extensions  appear in red. One should think of the dual of the $ku_*$ chart as an upside-down version of the chart. The dual of the element in position $(30,7)$ in Figure \ref{fig2} is in position $(34,0)$ in Figure \ref{fig2}.

\bigskip
\begin{minipage}{6in}
\begin{fig}\label{fig1}

{\bf A portion of $ku^*(K_2)$}

\begin{center}

\begin{\tz}[scale=.55]
\draw (0,0) -- (2,1) -- (2,0) -- (16,7);
\draw (-1,0) -- (17,0);
\draw (10,0) -- (16,3);
\draw (14,0) -- (16,1) -- (16,0);
\draw [color=red] (10,0) -- (10,4);
\draw [color=red] (14,0) -- (14,2);
\draw [color=red] (16,1) -- (16,3);
\draw [color=red] (12,1) -- (12,5);
\draw [color=red] (14,2) -- (14,6);
\draw [color=red] (16,3) -- (16,7);
\node at (0,-.8) {$36$};
\node at (4,-.8) {$32$};
\node at (8,-.8) {$28$};
\node at (12,-.8) {$24$};
\node at (16,-.8) {$20$};
\node at (0,0) {\sb};
\node at (2,0) {\sb};
\node at (2,1) {\sb};
\node at (4,1) {\sb};
\node at (6,2) {\sb};
\node at (8,3) {\sb};
\node at (10,4) {\sb};
\node at (12,5) {\sb};
\node at (14,6) {\sb};
\node at (16,7) {\sb};
\node at (10,0) {\sb};
\node at (12,1) {\sb};
\node at (14,2) {\sb};
\node at (16,3) {\sb};
\node at (14,0) {\sb};
\node at (16,0) {\sb};
\node at (16,1) {\sb};
\end{\tz}
\end{center}
\end{fig}
\end{minipage}

\bigskip

\bigskip
\begin{minipage}{6in}
\begin{fig}\label{fig2}

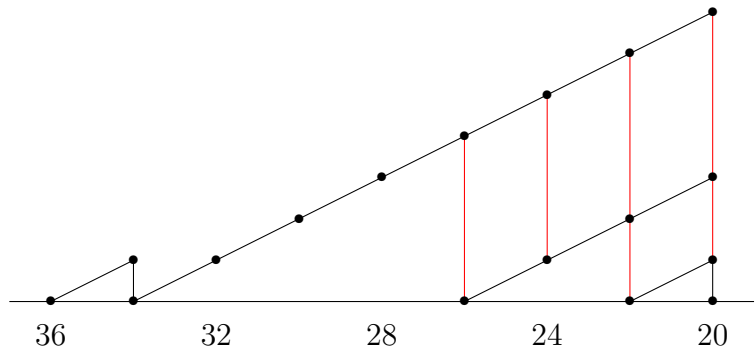
{\bf Corresponding portion of $ku_*(K_2)$}

\begin{center}

\begin{\tz}[scale=.55]
\draw (0,3) -- (0,0) -- (14,7);
\draw (0,2) -- (2,3) -- (2,1);
\draw (0,1) -- (6,4) -- (6,3);
\draw (4,2) -- (4,3);
\draw (14,0) -- (16,1);
\draw [color=red] (14,0) -- (14,7);
\node at (0,-.8) {$16$};
\node at (4,-.8) {$20$};
\node at (8,-.8) {$24$};
\node at (12,-.8) {$28$};
\node at (16,-.8) {$32$};
\draw (-1,0) -- (17,0);
\node at (0,0) {\sb};
\node at (0,1) {\sb};
\node at (0,2) {\sb};
\node at (0,3) {\sb};
\node at (2,1) {\sb};
\node at (2,2) {\sb};
\node at (2,3) {\sb};
\node at (4,2) {\sb};
\node at (4,3) {\sb};
\node at (6,3) {\sb};
\node at (6,4) {\sb};
\node at (8,4) {\sb};
\node at (10,5) {\sb};
\node at (12,6) {\sb};
\node at (14,7) {\sb};
\node at (14,0) {\sb};
\node at (16,1) {\sb};
\end{\tz}
\end{center}
\end{fig}
\end{minipage}

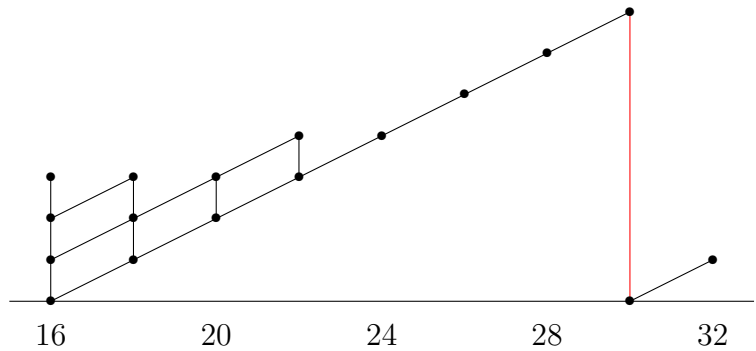
\def\line{\rule{.6in}{.6pt}}


\begin{thebibliography}{99}
\bibitem{BJ} A.Baker and A.Jeanneret, {\em Brave new Hopf algebroids and extensions of $MU$-algebras}, Homology, Homotopy, and Applications {\bf 4} (2002) 163--173.
\bibitem{DW} D.M.Davis and W.S.Wilson,  {\em Stiefel-Whitney classes and immersions of orientable and Spin manifolds}, Topology and Appl {\bf 307} (2022) https://doi.org/10.1016/j.topol.2021.107780.
\bibitem{DW2} \line, {\em Connective $K$-theory of the Eilenberg-MacLane space $K(\Z/2,2)$}, in preparation, www.lehigh.edu/$\sim$dmd1/kkpaper3.pdf.
\bibitem{Gr} J.P.C.Greenlees, {\em Gorenstein duality and universal coefficient theorems}, in preparation.
\bibitem{JW} D.C.Johnson and W.S.Wilson, {\em Projective dimension and Brown-Peterson homology}, Topology {\bf 12} (1973) 327--353.
\bibitem{MR} M.Mahowald and C.Rezk, {\em Brown-Comenetz duality and the Adams spectral sequence}, Amer Jour Math {\bf 121} (1999) 1153--1177.
\bibitem{Rob} C.A.Robinson, {\em Spectra of derived module homomorphisms}, Math Proc Camb Phil Soc {\bf 101} (1987) 249--257.
\bibitem{Wei} C.A.Weibel, {\em An introduction to homological algebra}, Cambridge studies in advanced mathematics {\bf 38} (1994).
\bibitem{W} W.S.Wilson, {\em A new relation on the Stiefel-Whitney classes of Spin manifolds}, Ill Jour Math {\bf 17} (1973) 115--127.
\end{thebibliography}
\end{document}